\documentclass[11pt,twoside]{article}
\usepackage{mathrsfs}
\usepackage{amssymb}
\usepackage{amsmath}
\usepackage[all]{xy}
\usepackage{amsthm}

\usepackage{hyperref}

\usepackage{url}
\numberwithin{equation}{section}

\newtheorem{theorem}{Theorem}[section]
\newtheorem{proposition}[theorem]{Proposition}
\newtheorem{lemma}[theorem]{Lemma}

\newtheorem{corollary}[theorem]{Corollary}

\newtheorem{theorem*}{Theorem}

\theoremstyle{definition}
\newtheorem{definition}[theorem]{Definition}

\theoremstyle{remark}
\newtheorem{remark}[theorem]{Remark}

\usepackage{paralist}

\numberwithin{equation}{section} \theoremstyle{plain}
\newtheorem*{thm*}{Main Theorem}

\newtheorem*{cor*}{Corollary}

\newtheorem*{lem*}{Lemma}

\newtheorem*{prop*}{Proposition}

\newtheorem*{rem*}{Remark}

\newtheorem*{exa*}{Example}

\newtheorem*{df*}{Definition}

\newtheorem*{conj*}{Conjecture}

\newtheorem*{Fa*}{Fact}

\newtheorem*{Qu*}{Question}
\newtheorem*{ack*}{ACKNOWLEDGEMENTS}

\newcommand{\pf}{\noindent\begin {proof}}
\newcommand{\epf}{\end{proof}}

\newcommand{\Ext}{\mbox{\rm Ext}}
\newcommand{\Hom}{\mbox{\rm Hom}}

\def\Im{\mathop{\rm Im}\nolimits}

\def\Coker{\mathop{\rm Coker}\nolimits}
\def\Tr{\mathop{\rm Tr}\nolimits}

\def\mod{\mathop{\rm mod}\nolimits}
\def\Mod{\mathop{\rm Mod}\nolimits}

\def\fd{\mathop{\rm fd}\nolimits}
\def\id{\mathop{\rm id}\nolimits}
\def\pd{\mathop{\rm pd}\nolimits}

\def\sup{\mathop{\rm sup}\nolimits}
\def\inf{\mathop{\rm inf}\nolimits}
\def\soc{\mathop{\rm soc}\nolimits}

\def\findim{\mathop{\rm fin.dim}\nolimits}

\def\Gpd{\mathop{\rm G\text{-}pd}\nolimits}

\def\Hom{\mathop{\rm Hom}\nolimits}
\def\Ext{\mathop{\rm Ext}\nolimits}
\def\sup{\mathop{\rm sup}\nolimits}

\def\lim{\mathop{\underrightarrow{\rm lim}}\nolimits}

\begin{document}
\begin{center}
{\large \bf Auslander-type Conditions and Weakly Gorenstein Algebras}
\footnote{The research was partially supported by NSFC (Grant Nos. 12371038, 12171207).}

\vspace{0.5cm}
Zhaoyong Huang\\
{\footnotesize\it Department of Mathematics, Nanjing University,
Nanjing 210093, Jiangsu Province, P.R. China}\\
{\footnotesize\it E-mail: huangzy@nju.edu.cn}
\end{center}

\bigskip
\centerline {\bf  Abstract}
\bigskip
\leftskip10truemm \rightskip10truemm \noindent
Let $R$ be an Artin algebra. Under certain Auslander-type conditions, we give some equivalent characterizations
of (weakly) Gorenstein algebras in terms of the properties of Gorenstein projective modules and modules satisfying
Auslander-type conditions. As applications, we provide some support for several homological conjectures.
In particular, we prove that if $R$ is left quasi Auslander, then
$R$ is Gorenstein if and only if it is (left and) right weakly Gorenstein; and that if $R$
satisfies the Auslander condition, then $R$ is Gorenstein if and only if it is left or right
weakly Gorenstein. This is a reduction of an Auslander--Reiten's conjecture, which states that $R$ is Gorenstein
if $R$ satisfies the Auslander condition. 
\vbox to 0.2cm{}
\leftskip0truemm \rightskip0truemm

\section{\bf Introduction}
A left and right Noetherian ring is called {\it Iwanaga--Gorenstein} ({\it Gorenstein} for short)
if its left and right self-injective dimensions are finite.
The fundamental theorem in \cite{B2} states that a commutative Noetherian ring $R$
is Gorenstein if and only if the flat dimension of the $i$-th term in a
minimal injective coresolution of $R$ as an $R$-module is at most
$i-1$ for any $i \geq 1$. In the non-commutative case, Auslander proved
that the latter condition is left-right symmetric \cite[Theorem 3.7]{FGR}; in
this case, $R$ is said to satisfy the {\it Auslander condition}.
Thus, the above result in \cite{B2} can be restated as follows: A commutative
Noetherian ring is Gorenstein if and only if it satisfies the Auslander condition.
Based on it, Auslander and Reiten \cite{AR1} conjectured that an Artin algebra
satisfying the Auslander condition is Gorenstein. We call this conjecture {\bf ARC} for short.
It is situated between the well known Nakayama conjecture and the generalized Nakayama conjecture
\cite[p.2]{AR1}. All these conjectures remain still open.

As a generalization of the notion of the Auslander condition, Huang and Iyama \cite{HI}
introduced the notion of Auslander-type conditions of rings as follows. For any
$m\geq 0$, a left and right Noetherian ring is said to be $G_{\infty}(m)$ if
for any finitely generated left $R$-module $M$ and $i\geq 1$,
it holds that $\Ext_{R^{op}}^{0\leq j\leq i-1}(X,R)=0$ for any right $R$-submodule $X$
of $\Ext_R^{i+m}(M,R)$; equivalently, if the
flat dimension of the $i$-th term in a minimal injective coresolution of $R_R$ is at most
$i+m-1$ for any $i\geq 1$ \cite[p.99]{HI}. Non-commutative rings satisfying Auslander-type conditions
are analogues of commutative Gorenstein rings. Such rings play a crucial role in homological algebra,
representation theory of algebras and non-commutative algebraic
geometry, see \cite{AR1,AR2,EHIS,FGR,H1,H2,HI,IS,Iy,Mi,S,W}
and references therein. Recently, we introduced
modules satisfying Auslander-type condition $G_{\infty}(m)$ for any $m\geq 0$ \cite{H5},
see Definition \ref{def-2.4} below.

As a generalization of the notion of Gorenstein algebras, Ringel and Zhang \cite{RZ} introduced that of weakly Gorenstein
algebras. Marczinzik \cite{M} posed the following question: Is a left weakly Gorenstein Artin algebra also right weakly
Gorenstein? For the sake of convenience, we state this question as the following conjecture.

\vspace{0.2cm}

{\bf Weakly-Gorenstein Symmetry Conjecture} ({\bf WGSC} for short): An Artin algebra is left weakly Gorenstein
if and only if it is right weakly Gorenstein.

\vspace{0.2cm}

It is related to the following famous conjecture.

\vspace{0.2cm}

{\bf Gorenstein Symmetry Conjecture} ({\bf GSC} for short): For an Artin algebra,
its left self-injective dimension is finite if and only if so is its right self-injective dimension.

\vspace{0.2cm}

Note that for a left and right Noetherian ring, its left and right self-injective dimensions coincide
if both of them are finite \cite[Lemma A]{Z}. Thus, an equivalent version of {\bf GSC}
is that for an Artin algebra, its left and right self-injective dimensions coincide.

It was proved that {\bf WGSC} implies {\bf GSC} \cite[p.33]{RZ}, and that {\bf GSC} holds true for Artin algebras satisfying
the Auslander condition \cite[Corollary 5.5(b)]{AR1}. We proved that an Artin algebra satisfying the Auslander condition
is Gorenstein if and only if the subcategory of finitely generated modules satisfying the Auslander condition is contravariantly
finite \cite[Theorem 5.8]{H5}. The aim of this paper is to give some equivalent characterizations of (weakly) Gorenstein
algebras under certain Auslander-type conditions, and then provide some support for these conjectures mentioned above.

The paper is organized as follows.
In Section 2, we give some terminology and preliminary results.
Let $R$ be an arbitrary ring. We use $\mathcal{GP}(\Mod R)$ to denote the category of
Gorenstein projective left $R$-modules. For any $m\geq 0$, we use $\mathcal{GP}(\Mod R)^{\leq m}$
to denote the category of left $R$-modules with Gorenstein projective dimension at most $m$,
and use $\mathcal{G}_{\infty}(m)$ to denote the category of left $R$-modules being $G_{\infty}(m)$.

In Section 3, $R$ is an arbitrary ring. We prove that any module in $\mathcal{G}_{\infty}(m)$
is isomorphic to a kernel (resp. a cokernel) of a homomorphism from a module with finite flat dimension
to certain syzygy module, and as a consequence we get that if a left $R$-module $M$ satisfies
the Auslander condition (that is, $M\in\mathcal{G}_{\infty}(0)$), then $M$ is an $\infty$-flat syzygy module,
and the converse holds true if $_RR$ satisfies the Auslander condition (Theorem \ref{thm-3.3}).
For any $m,s\geq 0$, we prove that $\mathcal{GP}(\Mod R)=\mathcal{G}_{\infty}(m)$
if and only if $\mathcal{GP}(\Mod R)^{\leq s}=\mathcal{G}_{\infty}(m+s)$
(Proposition \ref{prop-3.5}). We also prove that if $R$ is a Gorenstein ring,
then any module in $\mathcal{G}_{\infty}(m)$ has Gorenstein projective dimension at most $m$
(Theorem \ref{thm-3.6}).

In Section 4, $R$ is an Artin algebra. We get some equivalent characterizations for
$_RR\in\mathcal{G}_{\infty}(m)$ and $R$ being Gorenstein as follows. The case for $m=0$ in the following result
except the statement (2) has been obtained in \cite[Corollary 3.5]{WLWH}, which is the Gorenstein version of \cite[Theorem 5.9]{H5}.

\begin{theorem}\label{thm-1.1} {\rm (Theorem \ref{thm-4.6})}
Let $m\geq 0$. Then the following statements are equivalent.
\begin{enumerate}
\item[$(1)$] $_RR\in G_{\infty}(m)$ and $R$ is Gorenstein.
\item[$(2)$] $_RR\in G_{\infty}(m)$ and the left self-injective dimension of $R$ is finite.
\item[$(3)$] $\mathcal{GP}(\Mod R)\subseteq\mathcal{G}_{\infty}(m)\subseteq\mathcal{GP}(\Mod R)^{\leq m}$.
\item[$(4)$] $\mathcal{GP}(\Mod R)^{\leq s}\subseteq\mathcal{G}_{\infty}(m+s)\subseteq
\mathcal{GP}(\Mod R)^{\leq m+s}$ for any $s\geq 0$.
\item[$(i)_{f}$] The finitely generated version of $(i)$ with $i=3,4$.
\end{enumerate}
\end{theorem}

Under certain Auslander-type conditions, we get some equivalent characterizations of (weakly) Gorenstein algebras.

\begin{theorem}\label{thm-1.2} {\rm (Theorem \ref{thm-4.9})}
If $_RR\in\mathcal{G}_{\infty}(m)$ and $R_R\in\mathcal{G}_{\infty}(m')^{op}$
with $m,m'\geq 0$, then the following statements are equivalent.
\begin{enumerate}
\item[$(1)$] $R$ is Gorenstein.
\item[$(2)$] $R$ is left and right weakly Gorenstein.
\item[$(3)$] The left self-injective dimension of $R$ is finite.
\item[$(4)$] $R$ is left weakly Gorenstein.
\item[$(5)$] $\mathcal{GP}(\Mod R)$ coincides with the left orthogonal category of projective left $R$-modules.
\item[$(i)^{op}$] The opposite version of $(i)$ with $3\leq i\leq 5$.
\end{enumerate}
\end{theorem}

Furthermore, we consider algebras satisfying small Auslander-type conditions.
We prove that if $R$ is left quasi Auslander (that is, $_RR\in \mathcal{G}_{\infty}(1)$), then $R$ is Gorenstein if and only if
the left or right self-injective dimension of $R$ is finite, and if and only if $R$ is (left and) right weakly Gorenstein
(Theorem \ref{thm-4.10}). Moreover, we get some equivalent characterizations of Auslander--Gorenstein algebras (Theorem \ref{thm-4.11}),
which yields that if $R$ satisfies the Auslander condition (that is, $_RR\in \mathcal{G}_{\infty}(0)$),
then $R$ is Gorenstein if and only if $R$ is left or right weakly Gorenstein (Corollary \ref{cor-4.12}).

Consequently, we conclude that
\begin{enumerate}
\item[$(1)$] Over an Artin algebra $R$ satisfying $_RR\in\mathcal{G}_{\infty}(m)$ and $R_R\in\mathcal{G}_{\infty}(m')^{op}$
with $m,m'\geq 0$, both {\bf WGSC} and {\bf GSC} hold true (Theorem \ref{thm-1.2}).
\item[$(2)$] Over a left quasi Auslander Artin algebra, {\bf GSC} holds true,
but we do not know whether {\bf WGSC} holds true or not (Theorem \ref{thm-4.10}).
\item[$(3)$] Assume that an Artin algebra $R$ satisfies the Auslander condition (equivalently,
$_RR\in\mathcal{G}_{\infty}(0)$ and $R_R\in\mathcal{G}_{\infty}(0)^{op}$). Then both {\bf WGSC} and {\bf GSC} hold true for $R$
by putting $m=m'=0$ in Theorem \ref{thm-1.2}. Note that {\bf GSC} holds true for an Artin algebra $R$
satisfying the Auslander condition has been obtained in \cite[Corollary 5.5(b)]{AR1}.
Moreover, we have that $R$ is Gorenstein if and only if it is left or right weakly Gorenstein (Corollary \ref{cor-4.12}).
This is a reduction of {\bf ARC},
since Gorenstein algebras are left and right weakly Gorenstein, but the converse does not hold true in general \cite{M,R,RZ}.
\end{enumerate}

\section{\bf Preliminaries}

Throughout this paper, all rings are associative rings with unit and all modules are unital.
For a ring $R$, we use $\Mod R$ to denote the category of left $R$-modules, and use $\mod R$
to denote the category of finitely generated left $R$-modules. For a module $M\in\Mod R$,
we use $\pd_RM$, $\fd_RM$ and $\id_RM$ to denote the projective, flat and injective dimensions
of $M$, respectively.

Let $R$ be a ring. We write $(-)^*=\Hom(-,R)$. Let $M\in\Mod R$ and let $\sigma_M:M\to M^{**}$
via $\sigma_M(x)(f)=f(x)$ for any $x\in M$ and $f\in M^*$ be the canonical evaluation homomorphism.
Recall that $M$ is called {\it torsionless} if $\sigma_M$ is a monomorphism, and is called
{\it reflexive} if $\sigma_M$ is an isomorphism. Let
$$\cdots\to P_i\to\cdots\to P_1\to P_0\to M\to 0$$
and
$$0\to M\to E^0(M)\to E^1(M)\to\cdots\to E^i(M)\to\cdots$$
be a projective resolution and a minimal injective coresolution of $M$, respectively.
For any $n\geq 1$, $\Omega^n(M):=\Im(P_n\to P_{n-1})$ and $\Omega^{-n}(M):=\Im(E^{n-1}(M)\to E^n(M))$
are called the {\it $n$-syzygy} and {\it $n$-cosyzygy} of $M$, respectively. In particular, $\Omega^0(M)=M$.
Note that the $n$-syzygy of $M$ is defined up to projective summands. We write
$$\Omega^{n}(\Mod R):=\{M\in\Mod R\mid M\ \text{is an}\ n\text{-syzygy module}\}\ \text{for any}\ n\geq 1,$$
$$\Omega^{\infty}(\Mod R):=\cap_{n\geq 1}\Omega^{n}(\Mod R)\ \ \text{and}\ \
\Omega^{\infty}(\mod R):=\Omega^{\infty}(\Mod R)\cap\mod R.$$
For a subcategory $\mathcal{X}$ of $\Mod R$, we write
$${^{\bot}\mathcal{X}}:=\{M\in\Mod R\mid\Ext^{\geq 1}_{R}(M,X)=0\ \text{for any}\ X\in\mathcal{X}\},$$
and write ${^{\bot}X}:={^{\bot}\mathcal{X}}$ if $\mathcal{X}=\{X\}$.

Let $R$ be a left and right Noetherian ring and $M\in\mod R$, and let
$$P_1\buildrel{f}\over\longrightarrow P_0\to M\to 0$$
be a projective presentation of $M$ in $\mod R$. Recall from \cite{AB} that $\Tr M:=\Coker f^*$
is called the {\it transpose} of $M$. Note that the transpose of $M$ is defined up to projective summands \cite[p.51]{AB}.
A module $M\in\mod R$ is called {\it $\infty$-torsionfree} if $\Tr M\in{^{\bot}R_R}\cap\mod R^{op}$.
We write $$\mathcal{T}(\mod R):=\{M\in\mod R\mid M\ \text{is}\ \infty\text{-torsionfree}\}.$$
By \cite[Theorem 2.17]{AB}, we have $\mathcal{T}(\mod R)\subseteq\Omega^{\infty}(\mod R)$.

\begin{definition} \label{def-2.1} {\rm (\cite{AB})}.
Let $R$ be a left and right Noetherian ring. A module $M\in\mod R$ is said to
{\it have Gorentein dimension zero} if
$$\Ext^{\geq 1}_{R}(M,R)=0=\Ext^{\geq 1}_{R^{op}}(\Tr M,R);$$
equivalently, if $M$ is reflexive and
$$\Ext^{\geq 1}_{R}(M,R)=0=\Ext^{\geq 1}_{R^{op}}(M^*,R).$$
\end{definition}

Let $R$ be a ring. We write $\mathcal{P}(\Mod R):=\{$projective left $R$-modules$\}$.
Recall from \cite{EJ} that a module $M\in\Mod R$ is called {\it Gorenstein projective}
if there exists an exact sequence
$$\cdots\to P_1\to P_0\to P^0\to P^1\to\cdots$$
in $\Mod R$ with all $P_i,P^i$ in $\mathcal{P}(\Mod R)$, such that it remains exact after applying the
functor $\Hom_R(-,P)$ for any $P\in\mathcal{P}(\Mod R)$ and $M\cong\Im(P_0\to P^0)$.
We write $$\mathcal{GP}(\Mod R):=\{\text{Gorenstein projective left}\ R\text{-modules}\}\ \text{and}\
\mathcal{GP}(\mod R):=\mathcal{GP}(\Mod R)\cap\mod R.$$
It is well known that over a left and right noetherian ring, a finitely generated module has Gorenstein dimension zero
if and only if it is Gorenstein projective \cite{AM,EJ}, and thus
$$\mathcal{GP}(\mod R)=({^{\bot}{_RR}}\cap\mod R)\cap\mathcal{T}(\mod R).$$
Now, finitely generated modules having Gorenstein dimension zero over left and right noetherian rings
are usually referred to as Gorenstein projective modules.

For any $M\in\mod R$ (resp. $\mod R^{op}$),
it is well known that $M$ and $\Tr\Tr M$ are projectively equivalent. So we have the following observation.

\begin{lemma}\label{lem-2.2}
Let $R$ be a left and right Noetherian ring. Then
a module $M\in\mod R$ $($resp. $\mod R^{op})$ is Gorenstein projective if and only if so is $\Tr M$.
\end{lemma}

Recall from \cite{FGR} that a left and right Noetherian ring $R$ is said to satisfy the
{\it Auslander condition} if $\fd_RE^i(_RR)\leq i$ for any $i \geq 0$.
As a generalization of rings satisfying the Auslander condition,
Huang and Iyama \cite{HI} introduced the notion of rings satisfying
Auslander-type conditions, which was extended to that of modules
satisfying Auslander-type conditions as follows.

\begin{definition}\label{def-2.4} {\rm (\cite{H5})}
Let $R$ be a ring and let $m\geq 0$. A module $M\in\Mod R$ is said to be {\it $G_{\infty}(m)$} if
$\fd_RE^i(M)\leq i+m$ for any $i\geq 0$. In particular, $M$ is said to satisfy the
{\it Auslander condition} if it is $G_{\infty}(0)$.
\end{definition}

Let $R$ be a left and right Noetherian ring. Then $_RR$ is $G_{\infty}(m)$ if and only if the ring $R$ is $G_{\infty}(m)^{op}$
in the sense of \cite{HI} (cf. Introduction). Notice that the notion of the Auslander condition is left-right symmetric
\cite[Theorem 3.7]{FGR}, so the ring $R$ satisfies the Auslander condition if and only if
both $_RR$ and $R_R$ satisfy the Auslander condition.
However, in general, the notion of $R$ being $G_{\infty}(m)$ is not left-right symmetric when $m\geq 1$
\cite{AR2,HI}. It should be pointed out that modules satisfying Auslander-type conditions are ubiquitous.
For example, if $R$ is a left and right Noetherian ring and $\id_{R^{op}}R\leq m$, then any module
in $\Mod R$ is $G_{\infty}(m)$ \cite[Example 4.2(3)]{H5}. For more examples of modules satisfying Auslander-type conditions,
the reader is referred to \cite[Example 4.2]{H5}.

Let $\mathcal{X}$ be a subcategory of $\Mod R$ and $M\in \Mod R$. The \emph{$\mathcal{X}$-projective dimension}
$\mathcal{X}$-$\pd_RM$ of $M$ is defined as $\inf\{n\mid$ there exists an exact sequence
$$0 \to X_n \to \cdots \to X_1\to X_0 \to M\to 0$$
in $\Mod R$ with all $X_i$ in $\mathcal{X}\}$. If no such an integer exists, then set $\mathcal{X}$-$\pd_RM=\infty$.
For any $s\geq 0$, we write
$$\mathcal{X}^{\leq s}:=\{M\in\Mod R\mid\mathcal{X}\text{-}\pd_RM\leq s\}.$$
When $\mathcal{X}=\mathcal{GP}(\Mod R)$ or $\mathcal{GP}(\mod R)$,
the $\mathcal{X}$-projective dimension of $M$ is exactly the Gorenstein projective dimension $\Gpd_RM$ of $M$.

\section{\bf Syzygy modules and Gorenstein projective dimension}

In this section, $R$ is an arbitrary ring. For any $m\geq 0$, we write
$$\mathcal{G}_{\infty}(m):=\{M\in\Mod R\mid M\ \text{is}\ G_{\infty}(m)\}.$$
Then we have the following inclusion chain
$$\mathcal{G}_{\infty}(0)\subseteq\mathcal{G}_{\infty}(1)\subseteq\cdots\subseteq
\mathcal{G}_{\infty}(m)\subseteq\cdots.$$

\begin{lemma}\label{lem-3.1}
If $R$ is a left Noetherian ring and ${_RR}\in\mathcal{G}_{\infty}(m)$,
then any flat module in $\Mod R$ is in $\mathcal{G}_{\infty}(m)$.
\end{lemma}

\begin{proof}
It follows from \cite[Corollary 3.2]{H5}.
\end{proof}

The following lemma is used frequently in the sequel.

\begin{lemma}\label{lem-3.2}
Let
$$0\to M\to X^0\to X^1\to\cdots\to X^i\to\cdots\eqno{(3.1)}$$
be an exact sequence in $\Mod R$ and let $m\geq 0$. If $X^i\in\mathcal{G}_{\infty}(m)$
for any $i\geq 0$, then $M\in\mathcal{G}_{\infty}(m)$. In particular, the subcategory
$\mathcal{G}_{\infty}(m)$ is closed under kernels of epimorphisms.
\end{lemma}

\begin{proof}
By the exact sequence (3.1) and \cite[Corollary 3.9(1)]{H3}, we get the following exact sequence
$$0\to M\to E^0(X^0)\to E^1(X^0)\oplus E^0(X^1)\to\cdots\to\oplus_{i=0}^nE^{n-i}(X^i)\to\cdots.$$
For the reader's convenience, we give an outline of the construction of this exact sequence,
which is dual to that in the proof of \cite[Theorem 3.6]{H3}.
Put $M^i:=\Im(X^{i-1}\to X^i)$ for any $i\geq 1$. Let $n$ be an arbitrary positive integer.
We have an exact sequence $$0\to M^n\to E^0(X^n).\eqno{(3.2)}$$  From (3.2) and the exact sequence
$0\to X^{n-1}\to E^0(X^{n-1})\to E^1(X^{n-1})$, we obtain the following exact sequence
$$0\to M^{n-1}\to  E^0(X^{n-1})\to E^1(X^{n-1})\oplus E^0(X^n).\eqno{(3.3)}$$
Then from (3.3) and the exact sequence $0\to X^{n-2}\to E^0(X^{n-2})\to E^1(X^{n-2})\to E^2(X^{n-2})$,
we obtain the following exact sequence
$$0\to M^{n-2}\to  E^0(X^{n-2})\to E^1(X^{n-2})\oplus E^0(X^{n-1})\to E^2(X^{n-2})\oplus E^1(X^{n-1})\oplus E^0(X^{n}).$$
Continuing this process, we obtain the following exact sequence
$$0\to M\to E^0(X^0)\to E^1(X^0)\oplus E^0(X^1)\to\cdots\to\oplus_{i=0}^nE^{n-i}(X^i).$$
Then the desired exact sequence is obtained because of the arbitrariness of $n$.
Since $X^i\in G_{\infty}(m)$, we have $\fd_RE^{j}(X^i)\leq j+m$ for any $i,j\geq 0$.
So $\fd_R\oplus_{i=0}^nE^{n-i}(X^i)\leq n+m$ for any $n\geq 0$, and thus
$M\in G_{\infty}(m)$.
\end{proof}

For any $n\geq 1$, we write $\Omega^{n}_{\mathcal{F}}(\Mod R):=\{M\in\Mod R\mid$
there exists an exact sequence
$$0\to M\to F^0\to F^1\to\cdots\to F^{n-1}$$
in $\Mod R$ with all $F^i$ flat$\}$, and write
$\Omega^{\infty}_{\mathcal{F}}(\Mod R):=\cap_{n\geq 1}\Omega^{n}_{\mathcal{F}}(\Mod R)$.

The first assertion in the following result shows that any module in $\mathcal{G}_{\infty}(m)$
is isomorphic to a kernel (resp. a cokernel) of a homomorphism from a module with finite flat dimension
to certain syzygy module.

\begin{theorem}\label{thm-3.3}
It holds that
\begin{enumerate}
\item[$(1)$] Let $M\in\mathcal{G}_{\infty}(m)$ with $m\geq 0$. Then for any $n\geq 1$, there exists an exact sequence
$$0\to G_0\to X_0\to G_1\to X_1\to 0$$
in $\Mod R$ with $M\cong\Im(X_0\to G_1)$ such that the following conditions are satisfied.
\begin{enumerate}
\item[$(a)$] $\fd_RG_0\leq m-1$ and $\fd_RG_1\leq m$.
\item[$(b)$] $X_0\in\Omega_{\mathcal{F}}^{n}(\Mod R)$ and $X_1\in\Omega^{n-1}(\Mod R)$.
\end{enumerate}
\item[$(2)$] $\mathcal{G}_{\infty}(0)\subseteq\Omega^{\infty}_{\mathcal{F}}(\Mod R)$ with equality
if $R$ is a left Noetherian ring and ${_RR}\in\mathcal{G}_{\infty}(0)$.
\end{enumerate}
\end{theorem}

\begin{proof}
(1) 
Let $M\in\mathcal{G}_{\infty}(m)$ and $n\geq 1$. We have the following two exact and commutative diagrams:
$$\xymatrix{
& 0 \ar@{-->}[d] & 0 \ar@{-->}[d] & 0 \ar@{-->}[d] & & 0 \ar@{-->}[d] &  & \\
0 \ar@{-->}[r]& K \ar@{-->}[r] \ar@{-->}[d] & K_0 \ar@{-->}[r] \ar@{-->}[d] & K_1 \ar@{-->}[r] \ar@{-->}[d]
& \cdots \ar@{-->}[r]& K_{n-1} \ar@{-->}[r] \ar@{-->}[d] & 0 & \\
0 \ar@{-->}[r]& X \ar@{-->}[r] \ar@{-->}[d] & P_{0} \ar@{-->}[r] \ar@{-->}[d] & P_{1} \ar@{-->}[r] \ar@{-->}[d]
& \cdots \ar@{-->}[r]& P_{n-1} \ar@{-->}[r] \ar@{-->}[d] & \Omega^{-n}(M) \ar@{-->}[r] \ar@{==}[d] & 0 \\
0 \ar[r]& M \ar[r]\ar@{-->}[d]  & E^0(M) \ar[r]\ar@{-->}[d]  & E^1(M) \ar[r]\ar@{-->}[d]  & \cdots \ar[r]
& E^{n-1}(M) \ar[r]\ar@{-->}[d]  & \Omega^{-n}(M) \ar[r] & 0\\
& 0 & 0 & 0 & & 0 &  & }$$
and
$$\xymatrix{
& 0 \ar@{-->}[d] & 0 \ar@{-->}[d] & 0 \ar@{-->}[d]  & \\
0 \ar@{-->}[r]& K \ar@{-->}[r] \ar@{-->}[d] & K_0 \ar@{-->}[r] \ar@{-->}[d] & K'_1 \ar@{-->}[r] \ar@{-->}[d] & 0 \\
0 \ar@{-->}[r]& X \ar@{-->}[r] \ar@{-->}[d] & P_{0} \ar@{-->}[r] \ar@{-->}[d] & X_1 \ar@{-->}[r] \ar@{-->}[d] & 0\\
0 \ar[r]& M \ar[r]\ar@{-->}[d]  & E^0(M) \ar[r]\ar@{-->}[d]  & \Omega^{-1}(M) \ar[r]\ar@{-->}[d]  & 0\\
& 0 & 0 & 0  & }$$
with all $P_i$ projective in $\Mod R$ and $X_1:=\Im(P_0\to P_1)\in\Omega^{n-1}(\Mod R)$. Since $M\in\mathcal{G}_{\infty}(m)$,
we have $\fd_RE^i(M)\leq i+m$ for any $i\geq 0$, and thus $\fd_RK_i\leq i+m-1$ for any $1\leq i\leq n-1$. It follows from
the upper row in the first diagram that $\fd_RK'_1\leq m$.

Consider the following pull-back diagram (Diagram (3.1)):
$$\xymatrix{& & 0 \ar@{-->}[d] & 0 \ar[d]& &\\
& & K'_1 \ar@{==}[r] \ar@{-->}[d] & K'_1 \ar[d]& &\\
0 \ar@{-->}[r] & M \ar@{==}[d] \ar@{-->}[r] & G_1 \ar@{-->}[d] \ar@{-->}[r] & X_1 \ar[d] \ar@{-->}[r] & 0\\
0 \ar[r] & M \ar[r] & E^{0}(M) \ar[r] \ar@{-->}[d] & \Omega^{-1}(M) \ar[d] \ar[r] & 0 &\\
& & 0 & 0. & & }$$
From the middle column, we obtain $\fd_RG_1\leq m$, and there exists an exact sequence
$$0\to G_0\to F\to G_1\to 0$$
in $\Mod R$ with $F$ flat and $\fd_RG_0\leq m-1$. Consider the following pull-back diagram (Diagram (3.2)):
$$\xymatrix{ &   0 \ar@{-->}[d] & 0 \ar[d]& & \\
&  G_0 \ar@{==}[r] \ar@{-->}[d]& G_0   \ar[d]& \\
0 \ar@{-->}[r] &  X_0 \ar@{-->}[r] \ar@{-->}[d]& F \ar@{-->}[r] \ar[d]& X_1 \ar@{-->}[r] \ar@{==}[d] & 0 \\
0 \ar[r] &  M \ar[r] \ar@{-->}[d]& G_1 \ar[r] \ar[d]& X_1 \ar[r]  & 0 \\
& 0 & 0. & }$$
From the middle row, we obtain $X_0\in\Omega_{\mathcal{F}}^{n}(\Mod R)$.
Now splicing the middle row in Diagram (3.1) and the leftmost column in Diagram (3.2) we get the desired exact sequence.

(2) To prove $\mathcal{G}_{\infty}(0)\subseteq\Omega^{\infty}_{\mathcal{F}}(\Mod R)$,
it suffices to prove that if $M\in\mathcal{G}_{\infty}(0)$, then $M\in\Omega^{n}_{\mathcal{F}}(\Mod R)$ for any $n\geq 1$.
Let $M\in\mathcal{G}_{\infty}(0)$ and $n\geq 1$. By (1), there exists an exact sequence
$$0\to M\to F\to G_1\to 0$$
in $\Mod R$ with $F$ flat and $G_1\in\Omega^{n-1}(\Mod R)$, and so $M\in\Omega^n_{\mathcal{F}}(\Mod R)$.

Now assume that $R$ is a left Noetherian ring and ${_RR}\in\mathcal{G}_{\infty}(0)$.
Then any flat module in $\Mod R$ is in $\mathcal{G}_{\infty}(0)$ by Lemma \ref{lem-3.1},
and thus $\Omega^{\infty}_{\mathcal{F}}(\Mod R)\subseteq\mathcal{G}_{\infty}(0)$ by Lemma \ref{lem-3.2}.
\end{proof}

We need the following lemma.

\begin{lemma}\label{lem-3.4}
For any $m,s\geq 0$, we have
$$\mathcal{G}_{\infty}(m)^{\leq s}\subseteq\mathcal{G}_{\infty}(m+s)$$
with equality if $\mathcal{P}(\Mod R)\subseteq\mathcal{G}_{\infty}(0)$.
\end{lemma}

\begin{proof}
Let $M\in\mathcal{G}_{\infty}(m)^{\leq s}$ and
$$0\to X_s\to\cdots\to X_1\to X_0\to M\to 0$$
be an exact sequence in $\Mod R$ with all $X_i$ in $\mathcal{G}_{\infty}(m)$.
According to \cite[Corollary 3.5]{H3}, we get the following two exact sequences
$$0\to M\to E\to\oplus_{i=0}^sE^{i+1}(X_i)\to\oplus_{i=0}^sE^{i+2}(X_i)\to
\oplus_{i=0}^sE^{i+3}(X_i)\to\cdots,\eqno{(3.2)}$$
$$0\to E^s(X_0)\to E^{s-1}(X_0)\oplus E^s(X_1)\to\cdots\to\oplus_{i=1}^sE^{i-1}(X_i)
\to\oplus_{i=0}^sE^{i}(X_i)\to E\to 0.\eqno{(3.3)}$$
Since all $X_i$ are in $\mathcal{G}_{\infty}(m)$, we have $\fd_RE^j(X_i)\leq j+m$
for any $j\geq 0$ and $0\leq i\leq s$. Thus $\fd_R\oplus_{i=0}^sE^{i+j}(X_i)\leq j+m+s$
for any $j\geq 1$. By (3.3), we have that $E$ is a direct summand of $\oplus_{i=0}^sE^{i}(X_i)$
and $\fd_RE\leq m+s$. Therefore we obtain $M\in\mathcal{G}_{\infty}(m+s)$ by (3.2).

Now suppose $\mathcal{P}(\Mod R)\subseteq\mathcal{G}_{\infty}(0)$. We will prove
$\mathcal{G}_{\infty}(m+s)\subseteq\mathcal{G}_{\infty}(m)^{\leq s}$ by induction on $s$.
The case for $s=0$ follows trivially. Suppose $s\geq 1$ and $M\in\mathcal{G}_{\infty}(m+s)$. Let
$$0\to K\to P \to M \to 0$$
be an exact sequence in $\Mod R$ with $P$ projective. Since $P\in\mathcal{G}_{\infty}(0)$,
it follows from \cite[Proposition 4.12]{H5} that $K\in\mathcal{G}_{\infty}(m+s-1)$, and hence
$\mathcal{G}_{\infty}(m)$-$\pd K\leq s-1$ by the induction hypothesis. This implies
$\mathcal{G}_{\infty}(m)$-$\pd_RM\leq s$ and $M\in\mathcal{G}_{\infty}(m)^{\leq s}$.
\end{proof}

By Lemma \ref{lem-3.4}, we obtain the following result.

\begin{proposition}\label{prop-3.5}
If $\mathcal{P}(\Mod R)\subseteq\mathcal{G}_{\infty}(0)$, then it holds that
\begin{enumerate}
\item[$(1)$] $\mathcal{GP}(\Mod R)=\mathcal{G}_{\infty}(m)$ if and only if
$\mathcal{GP}(\Mod R)^{\leq s}=\mathcal{G}_{\infty}(m+s)$ for any $s\geq 0$.
\item[$(2)$] If $R$ is a left and right Noetherian ring, then
$\mathcal{GP}(\mod R)=\mathcal{G}_{\infty}(m)\cap\mod R$ if and only if
$\mathcal{GP}(\mod R)^{\leq s}=\mathcal{G}_{\infty}(m+s)\cap\mod R$ for any $s\geq 0$.
\end{enumerate}
\end{proposition}

About the condition $\mathcal{P}(\Mod R)\subseteq\mathcal{G}_{\infty}(0)$ in Proposition \ref{prop-3.5},
we remark that if $R$ is a left Noetherian ring, then this condition is satisfied if and only if
$_RR$ satisfied the Auslander condition by \cite[Theorem 4.9]{H5}, and that if $R$ is an Artin algebra,
then $\mathcal{P}(\Mod R)=\mathcal{G}_{\infty}(0)$ if and only if $R$ is Auslander-regular (that is, the algebra $R$
satisfies Auslander condition and the global dimension of $R$ is finite) \cite[Theorem 5.9]{H5}.

\begin{theorem}\label{thm-3.6} It holds that
\begin{enumerate}
\item[$(1)$] If $R$ is a Gorenstein ring, then $\mathcal{G}_{\infty}(m)\subseteq\mathcal{GP}(\Mod R)^{\leq m}$ for any $m\geq 0$.
\item[$(2)$] If $R$ is a left Noetherian ring and $\id_RR<\infty$, then $\mathcal{GP}(\Mod R)=\Omega^{\infty}(\Mod R)$.
\end{enumerate}
\end{theorem}

\begin{proof}
(1) Let $R$ be a Gorenstein ring with $\id_RR=\id_{R^{op}}R\leq n$, and let $M\in\mathcal{G}_{\infty}(m)$.
Then $\Gpd_RM\leq n$ by \cite[Theorem 12.3.1]{EJ}. It suffices to prove $\Gpd_RM\leq m$. The case for $n\leq m$ is trivial.
Now suppose $n>m$ and $t:=n-m$. Consider the following exact sequence
$$0 \to M \to E^0(M) \to E^1(M)\to \cdots \to E^{t-1}(M) \to K^t\to 0,$$
where $K^t:=\Im(E^{t-1}(M)\to E^t(M))$. By \cite[Theorem 12.3.1]{EJ} again, we have $\Gpd_RK^t\leq n(=t+m)$.
Since $M\in\mathcal{G}_{\infty}(m)$, we have $\pd_RE^i(M)\leq i+m$ for any $0\leq i\leq t-1$.
Then it is easy to get $\Gpd_RM\leq m$ by \cite[Theorem 3.2 and Remark 4.4(3)(a)]{H4}.

(2) It suffices to prove $\Omega^{\infty}(\Mod R)\subseteq\mathcal{GP}(\Mod R)$.
If $R$ is a left Noetherian ring and $\id_RR<\infty$, then $\id_RP<\infty$ for any
$P\in\mathcal{P}(\Mod R)$ by \cite[Theorem 1.1]{B1}. Assume that $M\in\Omega^{\infty}(\Mod R)$ and
$$0\to M\to P^0\to P^1\to\cdots\to P^i\to\cdots$$
is an exact sequence in $\Mod R$ with all $P^i$ in $\mathcal{P}(\Mod R)$. It is easy to see that
the kernel of each homomorphism in the above exact sequence is in ${^{\bot}\mathcal{P}(\Mod R)}$ by dimension shifting.
Thus $M\in\mathcal{GP}(\Mod R)$ and $\Omega^{\infty}(\Mod R)\subseteq\mathcal{GP}(\Mod R)$.
\end{proof}

\section{\bf (Weakly) Gorenstein algebras}

In this section, $R$ is an Artin algebra. Under certain Auslander-type conditions, we will give some equivalent characterizations
for $\id_RR<\infty$ as well as for (weakly) Gorenstein algebras. As applications, we give some partial answers to
some related homological conjectures.

\subsection{Auslander-type conditions}

For any $M\in\Mod R$ and $m\geq 0$, we write
$${^{\bot_{\geq m+1}}}M:=\{A\in\Mod R\mid\Ext^{\geq m+1}_{R}(A,M)=0\}.$$

\begin{lemma}\label{lem-4.1}
Let $M\in\mod R$ such that $\Omega^{\infty}(\mod R)\subseteq{^{\bot_{\geq m+1}}}M\cap\mod R$ for some $m\geq 0$.
If there exists some $n\geq 0$ such that $\pd_RE^i(M)\leq n$ for any $i\geq n+m+1$, then $\id_RM\leq n+m$.
\end{lemma}

\begin{proof}
Let $M\in\mod R$. 
Set $K^i:=\Im(E^{i-1}(M)\to E^{i}(M))$ for any $i\geq 1$. Since $\pd_RE^i(M)\leq n$ for any $i\geq n+m+1$,
by the horseshoe lemma we obtain the following exact and commutative diagram
$$\xymatrix{& 0 \ar[d] & 0 \ar[d] & 0 \ar[d] & & 0 \ar[d] &\\
0 \ar[r] & K_n^{n+m+1}\ar[d]\ar[r] & P_n^{n+m+1}\ar[d]\ar[r] & P_n^{n+m+2}\ar[d]\ar[r]
& \cdots\ar[r]& P_n^{n+m+i}\ar[d]\ar[r] & \cdots\\
0 \ar[r] & P_{n-1}\ar[d]\ar[r] & P_{n-1}^{n+m+1}\ar[d]\ar[r] & P_{n-1}^{n+m+2}\ar[d]\ar[r]
& \cdots\ar[r]& P_{n-1}^{n+m+i}\ar[d]\ar[r] & \cdots\\
&\vdots \ar[d] & \vdots\ar@{-->}[d] & \vdots \ar[d] & & \vdots \ar[d] &\\
0 \ar[r] & P_1\ar[d]\ar[r] & P_1^{n+m+1}\ar[d]\ar[r] & P_1^{n+m+2}\ar[d]\ar[r] & \cdots\ar[r]& P_1^{n+m+i}\ar[d]\ar[r] & \cdots\\
0 \ar[r] & P_0\ar[d]\ar[r] & P_0^{n+m+1}\ar[d]\ar[r] & P_0^{n+m+2}\ar[d]\ar[r] & \cdots\ar[r]& P_0^{n+m+i}\ar[d]\ar[r] & \cdots\\
0 \ar[r] & K^{n+m+1}\ar[d]\ar[r] & E^{n+m+1}(M)\ar[d]\ar[r] & E^{n+m+2}(M)\ar[d]\ar[r]
& \cdots\ar[r] & E^{n+m+i}(M)\ar[d]\ar[r] & \cdots\\
& 0  & 0  & 0 & & 0}$$
in $\mod R$ with all $P_j$ and $P_j^t$ projective. Then $K_n^{n+m+1}\in\Omega^{\infty}(\mod R)$, and thus
$K_n^{n+m+1}\in{^{\bot_{\geq m+1}}}M\cap\mod R$ by assumption.
It follows from the leftmost column in the above diagram that $K^{n+m+1}\in{^{\bot_{\geq n+m+1}}}M\cap\mod R$.
Now applying the functor $\Hom_R(K^{n+m+1},-)$ to the exact sequence
$$0\to M\to E^0(M)\to E^1(M)\to\cdots\to E^{n+m-1}(M)\to K^{n+m}\to 0$$
yields $\Ext_R^1(K^{n+m+1},K^{n+m})=0$. It implies that the exact sequence
$$0\to K^{n+m}\to E^{n+m}(M)\to K^{n+m+1}\to 0$$
splits and $K^{n+m}$ is a direct summand of $E^{n+m}(M)$. Thus $K^{n+m}$ is injective and $\id_RM\leq n+m$.
\end{proof}

\begin{remark}\label{rem-4.2}
The same argument as above essentially proves the following result:
Let $R$ be an arbitrary ring (not necessarily an Artin algebra) and let $M\in\Mod R$ such that
$\Omega^{\infty}(\Mod R)\subseteq{^{\bot_{\geq m+1}}}M$ for some $m\geq 0$.
If there exists some $n\geq 0$ such that $\pd_RE^i(M)\leq n$ for any $i\geq n+m+1$, then $\id_RM\leq n+m$.
\end{remark}

Recall from \cite{RZ} that an Artin algebra $R$ is called {\it left weakly Gorenstein} if
$\mathcal{GP}(\mod R)={^{\bot}{_RR}}\cap\mod R$. Symmetrically, the notion of
{\it right weakly Gorenstein algebras} is defined.

\begin{proposition}\label{prop-4.3}
\begin{enumerate}
\item[]
\item[$(1)$] Assume that there exists some $n,m\geq 0$ such that $\pd_RE^i(_RR)\leq n$ for any $i\geq n+m+1$.
If $\Omega^{\infty}(\mod R)\subseteq{^{\bot_{\geq m+1}}{_RR}}\cap\mod R$, then $\id_RR\leq n+m$.
\item[$(2)$] Assume that there exists some $n\geq 0$ such that $\pd_RE^i(_RR)\leq n$ for any $i\geq n+1$.
If $R$ is right weakly Gorenstein and $\Omega^{\infty}(\mod R)=\mathcal{T}(\mod R)$,
then $\id_RR\leq n$.
\end{enumerate}
\end{proposition}

\begin{proof}
(1) Putting $M={_RR}$ in Lemma \ref{lem-4.1}, the assertion follows.

(2) Let $M\in\Omega^{\infty}(\mod R)$. Then $M\in\mathcal{T}(\mod R)$ by assumption,
and so $\Tr M\in{^{\bot}R_R}\cap\mod R^{op}$. Since $R$ is right weakly Gorenstein by assumption,
we have $\Tr M\in{^{\bot}R_R}\cap\mod R^{op}=\mathcal{GP}(\mod R^{op})$. Thus
$M\in\mathcal{GP}(\mod R)\subseteq{^{\bot}{_RR}}\cap\mod R$ by Lemma \ref{lem-2.2}.
This shows $\Omega^{\infty}(\mod R)\subseteq{^{\bot}{_RR}}\cap\mod R$, and then the assertion
follows from (1).
\end{proof}

The following lemma shows that all modules satisfying certain Auslander-type condition over an Artin algebra
satisfy the condition about projective dimension in Lemma \ref{lem-4.1}.

\begin{lemma}\label{lem-4.4}
If $M\in\mathcal{G}_{\infty}(m)$ $($resp. $N\in\mathcal{G}_{\infty}(m)^{op})$ with $m\geq 0$, then there exists
some $n\geq 0$ such that $\pd_RE^i(M)$ $($resp. $\pd_{R^{op}}E^i(N))\leq n$ for any $i\geq 0$.
\end{lemma}

\begin{proof}
Since $R$ is an Artin algebra, there exist only finitely many non-isomorphic indecomposable injective left
$R$-modules. Let $M\in\mathcal{G}_{\infty}(m)$. Without loss of generalization,
suppose that $\{E^0,\cdots,E^t\}$ is the complete set of
non-isomorphic indecomposable injective left modules that occur as direct summands of all $E^i(M)$.
Then there exists some $n\geq 0$ such that $\pd_RE^i\leq n$ for any
$1\leq i\leq t$, and thus $\pd_RE^i(M)\leq n$ for any $i\geq 0$. Symmetrically, if
$N\in\mathcal{G}_{\infty}(m)^{op}$, then there exists
some $n\geq 0$ such that $\pd_{R^{op}}E^i(N)\leq n$ for any $i\geq 0$.
\end{proof}

As a consequence, we obtain the following result.

\begin{proposition}\label{prop-4.5}
If $\mathcal{GP}(\mod R)=\mathcal{G}_{\infty}(m)\cap\mod R$ for some $m\geq 0$, then $\id_RR<\infty$.
\end{proposition}

\begin{proof}
Since $_RR\in\mathcal{GP}(\mod R)$, we have $_RR\in\mathcal{G}_{\infty}(m)$ by assumption,
It follows from Lemma \ref{lem-4.4} that $\pd_RE^i(_RR)\leq n$ for any $i\geq 0$.
Since any projective module in $\mod R$ is in $\mathcal{G}_{\infty}(m)$, we have
\begin{align*}
& \Omega^{\infty}(\mod R)\subseteq \mathcal{G}_{\infty}(m)\cap\mod R\ \ \text{(by Lemma \ref{lem-3.2})}\\
&\ \ \ \ \ \ \ \ \ \ \ \ \ \ \ =\mathcal{GP}(\mod R)\ \ \text{(by assumption)}\\
&\ \ \ \ \ \ \ \ \ \ \ \ \ \ \ \subseteq{^{\bot}{_RR}}\cap\mod R.
\end{align*}
Thus $\id_RR\leq n$ by Proposition \ref{prop-4.3}(1).
\end{proof}

We are now in a position to prove the following result, in which assertions (5) and (6) are
finitely generated versions of (3) and (4), respectively.

\begin{theorem}\label{thm-4.6}
For any $m\geq 0$, the following statements are equivalent.
\begin{enumerate}
\item[$(1)$] $_RR\in\mathcal{G}_{\infty}(m)$ and $R$ is Gorenstein.
\item[$(2)$] $_RR\in\mathcal{G}_{\infty}(m)$ and $\id_RR<\infty$.
\item[$(3)$] $\mathcal{GP}(\Mod R)\subseteq\mathcal{G}_{\infty}(m)\subseteq\mathcal{GP}(\Mod R)^{\leq m}$.
\item[$(4)$] $\mathcal{GP}(\Mod R)^{\leq s}\subseteq\mathcal{G}_{\infty}(m+s)\subseteq
\mathcal{GP}(\Mod R)^{\leq m+s}$ for any $s\geq 0$.
\item[$(5)$] $\mathcal{GP}(\mod R)\subseteq\mathcal{G}_{\infty}(m)\cap\mod R\subseteq\mathcal{GP}(\mod R)^{\leq m}$.
\item[$(6)$] $\mathcal{GP}(\mod R)^{\leq s}\subseteq\mathcal{G}_{\infty}(m+s)\cap\mod R
\subseteq\mathcal{GP}(\mod R)^{\leq m+s}$ for any $s\geq 0$.
\end{enumerate}
\end{theorem}

\begin{proof}
The implications $(1)\Longrightarrow (2)$, $(4)\Longrightarrow (3)\Longrightarrow (5)$ and
$(4)\Longrightarrow (6)\Longrightarrow (5)$ are trivial. By the symmetric version of \cite[Corollary 3]{H1},
we get $(2)\Longrightarrow (1)$.

$(1)\Longrightarrow (3)$ Since $R$ is Gorenstein by (1), we have $\mathcal{G}_{\infty}(m)
\subseteq\mathcal{GP}(\Mod R)^{\leq m}$ by Theorem \ref{thm-3.6}(1). On the other hand, since
$_RR\in\mathcal{G}_{\infty}(m)$ by (1), we have $\mathcal{P}(\Mod R)\subseteq\mathcal{G}_{\infty}(m)$
by Lemma \ref{lem-3.1}, and thus
$$\mathcal{GP}(\Mod R)\subseteq\Omega^{\infty}(\Mod R)\subseteq\mathcal{G}_{\infty}(m)$$
by Lemma \ref{lem-3.2}.

$(5)\Longrightarrow (2)$ Since any projective module in $\mod R$ is in $\mathcal{G}_{\infty}(m)\cap\mod R$ by (5), we have
$$\Omega^{\infty}(\mod R)\subseteq \mathcal{G}_{\infty}(m)\cap\mod R\subseteq\mathcal{GP}(\mod R)^{\leq m}
\subseteq{^{\bot_{\geq m+1}}{_RR}}\cap\mod R$$
by Lemma \ref{lem-3.2} and (5). Since $_RR\in\mathcal{G}_{\infty}(m)\cap\mod R$, there exists some $n\geq 0$ such that
$\pd_RE^i(_RR)\leq n$ for any $i\geq 0$ by Lemma \ref{lem-4.4}, and thus $\id_RR\leq n+m$ by Proposition \ref{prop-4.3}(1).

$(1)+(3)\Longrightarrow (4)$ Let $s\geq 0$. Since $\mathcal{GP}(\Mod R)\subseteq\mathcal{G}_{\infty}(m)$ by (3), we have
$$\mathcal{GP}(\Mod R)^{\leq s}\subseteq\mathcal{G}_{\infty}(m)^{\leq s}\subseteq\mathcal{G}_{\infty}(m+s)$$
by Lemma \ref{lem-3.4}. Since $R$ is Gorenstein by (1), we have $\mathcal{G}_{\infty}(m+s)\subseteq\mathcal{GP}(\Mod R)^{\leq m+s}$
by Theorem \ref{thm-3.6}(1).
\end{proof}

We need the following result.

\begin{proposition}\label{prop-4.7}
If $\id_RR<\infty$, then $R$ is right weakly Gorenstein.
The converse holds true if one of the following conditions is satisfied.
\begin{enumerate}
\item[$(1)$] $_RR\in\mathcal{G}_{\infty}(1)$.
\item[$(2)$] $_RR\in\mathcal{G}_{\infty}(m)$ and $R_R\in\mathcal{G}_{\infty}(m')^{op}$ for some $m,m'\geq 0$.
\end{enumerate}
\end{proposition}

\begin{proof}
The former assertion follows from the symmetric versions of \cite[Lemma 3.4]{HT} and \cite[Theorem 1.2]{RZ}.

Conversely, since $_RR\in\mathcal{G}_{\infty}(1)$ or $_RR\in\mathcal{G}_{\infty}(m)$ with $m\geq 0$ by assumption,
it follows from Lemma \ref{lem-4.4} that there exists some $n\geq 0$ such that $\pd_RE^i(_RR)\leq n$ for any $i\geq 0$.
When $_RR\in\mathcal{G}_{\infty}(1)$, we have $\Omega^{\infty}(\mod R)=\mathcal{T}(\mod R)$
by \cite[Proposition 1.6(a)]{AR2} and the symmetric version of \cite[Theorem 0.1]{AR2};
when $R_R\in\mathcal{G}_{\infty}(m')^{op}$ with $m'\geq 0$, that is, the algebra $R$ is $G_{\infty}(m')$,
we also have $\Omega^{\infty}(\mod R)=\mathcal{T}(\mod R)$
by \cite[Theorem 3.4]{HI}. Thus $\id_RR\leq n$ by Proposition \ref{prop-4.3}(2).
\end{proof}

The following corollary was proved in \cite[p.33]{RZ}, we give it a shorter proof.

\begin{corollary}\label{cor-4.8}
{\bf WGSC} implies {\bf GSC}.
\end{corollary}

\begin{proof}
Suppose that {\bf WGSC} holds true. If $\id_RR=n<\infty$, then $R$ is right weakly Gorenstein by
Proposition \ref{prop-4.7}, and hence is left weakly Gorenstein. It follows that
any $n$-syzygy module in $\mod R$ is in ${^{\bot}{_RR}}\cap\mod R=\mathcal{GP}(\mod R)$. So
$\Gpd_RM\leq n$ for any $M\in\mod R$, and hence $R$ is $n$-Gorenstein (that is, $\id_RR=\id_{R^{op}}R\leq n$)
by \cite[Theorem 12.3.1]{EJ}. Symmetrically, we have that if $\id_{R^{op}}R=n<\infty$, then $R$ is $n$-Gorenstein.
Thus {\bf GSC} holds true.
\end{proof}

The following result shows that the Gorensteinness and weakly Gorensteinness of an Artin algebra
are equivalent under certain Auslander-type conditions. It also shows that both {\bf GSC} and {\bf WGSC}
hold true for an Artin algebra $R$ such that $_RR$ and $R_R$ satisfy certain
Auslander-type conditions.

\begin{theorem}\label{thm-4.9}
If $_RR\in\mathcal{G}_{\infty}(m)$ and $R_R\in\mathcal{G}_{\infty}(m')^{op}$
with $m,m'\geq 0$, then the following statements are equivalent.
\begin{enumerate}
\item[$(1)$] $R$ is Gorenstein.
\item[$(2)$] $R$ is left and right weakly Gorenstein.
\item[$(3)$] $\id_RR<\infty$.
\item[$(4)$] $R$ is left weakly Gorenstein.
\item[$(5)$] $\mathcal{GP}(\Mod R)={^{\bot}{\mathcal{P}(\Mod R)}}$.
\item[$(i)^{op}$] Opposite version of $(i)$ with $3\leq i\leq 5$.
\end{enumerate}
\end{theorem}

\begin{proof}
It is trivial that $(5)\Longrightarrow (4)$ and $(2)\Longrightarrow (4)$.
By Proposition \ref{prop-4.7} and its symmetric version, we have $(1)\Longrightarrow (2)$
and $(3)\Longleftrightarrow (4)^{op}$. By Theorem \ref{thm-4.6} and its symmetric version,
we have $(1)\Longleftrightarrow (3)\Longleftrightarrow (3)^{op}$.
By \cite[Corollary 11.5.3]{EJ}, we have $(1)\Longrightarrow (5)$.

By symmetry, the proof is finished.
\end{proof}

\subsection{Small Auslander-type conditions}

Recall from \cite{H2} that $R$ is called {\it left quasi Auslander} if $_RR\in\mathcal{G}_{\infty}(1)$.
Compare the following result with Theorem \ref{thm-4.9}.

\begin{theorem}\label{thm-4.10}
Let $R$ be a left quasi Auslander algebra. Then the following statements are equivalent.
\begin{enumerate}
\item[$(1)$] $R$ is Gorenstein.
\item[$(2)$] $\id_RR<\infty$.
\item[$(3)$] $\id_{R^{op}}R<\infty$.
\item[$(4)$] $R$ is left and right weakly Gorenstein.
\item[$(5)$] $R$ is right weakly Gorenstein.
\item[$(6)$] $\mathcal{GP}(\Mod R^{op})={^{\bot}{\mathcal{P}(\Mod R^{op})}}$.
\end{enumerate}
\end{theorem}

\begin{proof}
It is trivial that $(4)\Longrightarrow (5)$ and $(6)\Longrightarrow (5)$.

By Proposition \ref{prop-4.7} and its symmetric version, we have $(1)\Longleftrightarrow (4)$.
By \cite[Corollary 4]{H1}, we have $(1)\Longleftrightarrow (2)\Longleftrightarrow (3)$.
By Proposition \ref{prop-4.7}(1), we have $(2)\Longleftrightarrow (5)$.
By \cite[Corollary 11.5.3]{EJ}, we have $(1)\Longrightarrow (6)$.
\end{proof}

Theorem \ref{thm-4.10} means that over a left quasi Auslander Artin algebra, {\bf GSC} holds true,
but we do not know whether {\bf WGSC} holds true or not.

Recall that $R$ is called {\it Auslander--Gorenstein} if $R$ satisfies the Auslander condition
and $R$ is Gorenstein. In the following result, assertions (5)--(7) are finitely generated versions
of (2)--(4) respectively.

\begin{theorem}\label{thm-4.11}
The following statements are equivalent.
\begin{enumerate}
\item[$(1)$] $R$ is Auslander--Gorenstein.
\item[$(2)$] $R$ satisfies the Auslander condition and $\mathcal{GP}(\Mod R)={^{\bot}{\mathcal{P}(\Mod R)}}$.
\item[$(3)$] $\mathcal{GP}(\Mod R)=\mathcal{G}_{\infty}(0)$.
\item[$(4)$] $\mathcal{GP}(\Mod R)^{\leq s}=\mathcal{G}_{\infty}(s)$ for any $s\geq 0$.
\item[$(5)$] $R$ satisfies the Auslander condition and $R$ is left weakly Gorenstein.
\item[$(6)$] $\mathcal{GP}(\mod R)=\mathcal{G}_{\infty}(0)\cap\mod R$.
\item[$(7)$] $\mathcal{GP}(\mod R)^{\leq s}=\mathcal{G}_{\infty}(s)\cap\mod R$ for any $s\geq 0$.
\item[$(i)^{op}$] Opposite version of $(i)$ with $2\leq i\leq 7$.
\end{enumerate}
\end{theorem}

\begin{proof}
The implications $(2)\Longrightarrow (5)$, $(3)\Longrightarrow (6)$ and $(4)\Longrightarrow (7)$  are trivial.

Note that $R$ satisfies the Auslander condition if and only if ${_RR}\in\mathcal{G}_{\infty}(0)$
and ${R_R}\in\mathcal{G}_{\infty}(0)^{op}$, and if and only if ${_RR}\in\mathcal{G}_{\infty}(0)$
or ${R_R}\in\mathcal{G}_{\infty}(0)^{op}$.
The implications $(3)\Longleftrightarrow (4)$ and $(6)\Longleftrightarrow (7)$
follow from Proposition \ref{prop-3.5}(1)(2) respectively.
The implication $(6)\Longrightarrow (1)$ follows from Proposition \ref{prop-4.5} and \cite[Corollary 5.5(b)]{AR1}.
The implications $(1)\Longleftrightarrow (3)$ and $(1)\Longleftrightarrow (2)\Longleftrightarrow (5)$
follow from Theorems \ref{thm-4.6} and \ref{thm-4.9}, respectively.

By symmetry, the proof is finished.
\end{proof}

Let $M$ be an $R$-module. An injective coresolution
$$0\to M\to E^0\buildrel{\delta^1}\over\longrightarrow E^1\buildrel{\delta^2}\over\longrightarrow\cdots
\buildrel{\delta^n}\over\longrightarrow E^n\buildrel{\delta^{n+1}}\over\longrightarrow\cdots$$
is called {\it ultimately closed} if there exists some $n$ such that $\Im\delta^n=\oplus W_j$
with each $W_j$ isomorphic to a direct summand of some $\Im\delta_{i_j}$ with $i_j<n$.
It is clear that a left $R$-module $M$ has an ultimately closed injective coresolution if $\id_RM<\infty$.
An algebra $R$ is said to be of {\it ultimately closed type} if the minimal injective coresolution of
any finitely generated left $R$-module is ultimately closed \cite{T}.
The class of algebras of ultimately closed type includes:
(1) Artin algebras with finite global dimension; (2) Artin algebras with radical square zero;
(3) Representation-finite algebras; (4) Artin algebras $R$ with Loewy length $m$ such that
$R/J^{m-1}$ is representation-finite, where $J$ is the Jacobson radical of $R$ \cite[p.110]{T}.

Recall from \cite{R} that $R$ is called {\it torsionless-finite} if there exists only
finitely many isomorphism classes of indecomposable torsionless modules in $\mod R$.
We claim that any torsionless-finite algebra is of ultimately closed type. Let $R$ be a torsionless-finite algebra.
It follows from \cite[Corollary 2.2]{R} that $R^{op}$ is also a torsionless-finite algebra and there exists only
finitely many isomorphism classes of indecomposable torsionless modules in $\mod R^{op}$. Using the usual duality
between $\mod R$ and $\mod R^{op}$ yields that there exists only
finitely many isomorphism classes of indecomposable 1-cosyzygy modules in $\mod R$. Thus $R$ is of ultimately closed type.
The claim is proved. The class of torsionless-finite algebras includes:
(1) Artin algebras $R$ with $R/\soc(R_R)$ representation-finite, where $\soc(R_R)$ is the socle of $R_R$;
(2) Minimal representation-infinite algebras;
(3) Artin algebras stably equivalent to hereditary algebras;
(4) Left or right glued algebras; and (5) Special biserial algebras without indecomposable projective-injective
modules \cite[Section 5]{R}.

Note that algebras $R$ such that $_RR$ has an ultimately closed injective coresolution
(particularly, algebras $R$ of ultimately closed type) are
right weakly Gorenstein algebras by the symmetric versions of \cite[Theorem 2.4]{HT} and \cite[Theorem 1.2]{RZ}.
However, such algebras are not Gorenstein in general, thus the following result can be regarded as a reduction of {\bf ARC}.

\begin{corollary}\label{cor-4.12}
If $R$ satisfies the Auslander condition, then the following statements are equivalent.
\begin{enumerate}
\item[$(1)$] $R$ is Gorenstein.
\item[$(2)$] $R$ is left or right weakly Gorenstein. 
\item[$(3)$] $R$ is left and right weakly Gorenstein.
\item[$(4)$] $\mathcal{GP}(\mod R)=\mathcal{T}(\mod R)$.
\item[$(5)$] $\mathcal{GP}(\mod R)=\mathcal{T}(\mod R)={^{\bot}{_RR}}\cap\mod R$.
\end{enumerate}
\end{corollary}

\begin{proof}
Since $R$ satisfies the Auslander condition, we have
$$\mathcal{G}_{\infty}(0)\cap\mod R=\Omega^{\infty}(\mod R)=\mathcal{T}(\mod R)$$
by \cite[Lemma 5.7]{H5}. Now the assertion follows from Theorem \ref{thm-4.11}.
\end{proof}

As indicated above, if $_RR$ (resp. $R_R$) has an ultimately closed injective coresolution,
then $R$ is a right (resp. left) weakly Gorenstein algebra.
As a consequence of Corollary \ref{cor-4.12}, we obtain the following result.

\begin{corollary}\label{cor-4.13}
{\bf ARC} holds true for the following classes of algebras $R$.
\begin{enumerate}
\item[$(1)$] $_RR$ or $R_R$ has an ultimately closed injective coresolution.
\item[$(2)$] Algebras of ultimately closed type.
\end{enumerate}
\end{corollary}

In the following, we give an alternative proof of Corollary \ref{cor-4.13}, which is independent of
Corollary \ref{cor-4.12}. Recall that the {\it finitistic dimension} $\findim R$ of $R$ is defined as
$$\findim R:=\sup\{\pd_RM\mid M\in\mod R\ \text{with}\ \pd_RM<\infty\}.$$
Assume that $R$ satisfies the Auslander condition. Then $R^{op}$ also satisfies the Auslander condition
by \cite[Theorem 3.7]{FGR}. It follows from \cite[Corollary 5.5(b)]{AR1} that $\id_RR<\infty$ if and only if
$\id_{R^{op}}R<\infty$, and hence $\id_RR=\id_{R^{op}}R$ by \cite[Lemma A]{Z}. Then
$$\findim R=\id_RR=\id_{R^{op}}R=\findim R^{op}\eqno{(4.1)}$$
by \cite[Corollary 5.3(1)]{HI}.
When $R_R$ has an ultimately closed injective coresolution, it is known from \cite[p.2983]{FW} that
this injective coresolution has a strongly redundant image in the sense of \cite{FW}. Then applying
\cite[Theorem 3]{FW} yields $\findim R^{op}<\infty$. Symmetrically, when $_RR$ has
an ultimately closed injective coresolution (particularly, when $R$ is of ultimately closed type),
we have $\findim R<\infty$. So, in both cases, we have $\id_RR=\id_{R^{op}}R<\infty$ by (4.1),
that is, $R$ is Gorenstein. Consequently we conclude that {\bf ARC} holds true for $R$.

\vspace{0.5cm}

{\bf Acknowledgements.} The author thanks the referees for very helpful and detailed suggestions.

\end{document}